\documentclass{amsart}
\usepackage{relsize}

\newcommand{\gopr}{\gnecessary}

\DeclareSymbolFont{extraup}{U}{zavm}{m}{n}
\DeclareMathSymbol{\vardiamond}{\mathalpha}{extraup}{87}
 \DeclareSymbolFont{symbolsC}{U}{txsyc}{m}{n}
\DeclareMathSymbol{\strictif}{\mathrel}{symbolsC}{74}
\DeclareMathSymbol{\strictfi}{\mathrel}{symbolsC}{75}
\DeclareMathSymbol{\strictiff}{\mathrel}{symbolsC}{76}
\newcommand{\tto}{\strictif}

\renewcommand{\descriptionlabel}[1]%
{\hspace{\labelsep}\emph{#1:}}

\usepackage{amssymb}
\usepackage{latexsym}
\usepackage{amsmath}
\usepackage{enumerate}
\usepackage{amsthm}
\usepackage{wasysym}
\usepackage{stmaryrd}
\usepackage{mathrsfs}
\usepackage{pifont}
\usepackage{tikz}
\usepackage[f]{esvect}
\usepackage[normalem]{ulem}

\newcommand{\qee} {\hspace*{2mm}\hfill \ding{109}}

\renewcommand{\iff}{\leftrightarrow}
\renewcommand{\leq}{\leqslant}
\renewcommand{\geq}{\geqslant}
\renewcommand{\preceq}{\preccurlyeq}

\renewcommand{\phi}{\varphi}

\renewcommand{\Theta}{\varTheta}
\renewcommand{\Phi}{\varPhi}
\renewcommand{\Psi}{\varPsi}
\renewcommand{\Xi}{\varXi}
\renewcommand{\Omega}{\varOmega}
\renewcommand{\Gamma}{\varGamma}

\newtheorem{theorem}{Theorem}[section]
\newtheorem{define}[theorem]{Definition}

\newtheorem{exa}[theorem]{Example}

\newtheorem{exerc}[theorem]{Exercise}

\newtheorem{conj}[theorem]{Conjecture}

\newtheorem{ques}[theorem]{Open Question}
\newenvironment{question}{\begin{ques} \rm}{\qee\end{ques}}

\newtheorem{rem}[theorem]{Remark}
\newenvironment{remark}{\begin{rem} \rm}{\qee\end{rem}}

\DeclareMathOperator{\possible}{\text{\tikz[scale=.6ex/1cm,baseline=-.6ex,rotate=45,line width=.1ex]{
                            \draw (-1,-1) rectangle (1,1);}}}
\DeclareMathOperator{\necessary}{\text{\tikz[scale=.6ex/1cm,baseline=-.6ex,line width=.1ex]{
                            \draw (-1,-1) rectangle (1,1);}}}

 \DeclareMathOperator{\gnecessary}{\text{\tikz[scale=.6ex/1cm,baseline=-.6ex,line width=.1ex]{
                            \draw[gray, fill = gray, fill opacity = .90] (-1,-1) rectangle (1,1);}}}

                            \definecolor{uuxred}{cmyk}{0.2,1,0.9,0.1}

\newcommand{\gnum}[1]{{\ulcorner #1 \urcorner}}
\newcommand{\mc}[1]{\mathcal #1}
\newcommand{\num}[1]{{\underline {#1}}}
\newcommand{\mf}[1]{{\mathfrak {#1}}}
\newcommand{\verz}[1]{\{ #1 \}}
\newcommand{\tupel}[1]{{\langle #1 \rangle}}
\newcommand{\apr}{{\vartriangle}}
\newcommand{\aco}{{\triangledown}}
\newcommand{\opr}{\necessary}
\newcommand{\oco}{\possible}

\newcommand{\bleq}{\mathbin{\leq}}
\newcommand{\bgeq}{\mathbin{\geq}}
\newcommand{\qedright}{\belowdisplayskip=-12pt}

\title{On a Question of Hamkins'}
\author{Albert Visser}
 \address{Philosophy, Faculty of Humanities,
                Utrecht University,
               Janskerkhof 13,
                3512BL~~Utrecht, The Netherlands}
\email{a.visser@uu.nl}
\date{\today}



\begin{document}
\begin{abstract}
Joel Hamkins asks whether there is a $\Pi^0_1$-formula $\rho(x)$ such that
$\rho(\gnum \phi)$ is independent over ${\sf PA}+\phi$, if this theory is consistent, where
this construction is extensional in $\phi$ with respect to {\sf PA}-provable equivalence.
We show that there can be no such extensional Rosser formula of any complexity.

We give a positive answer to Hamkins' question for the case where we replace
Extensionality by a weaker demand that we call \emph{Conditional Extensionality}. 
For this case, we prove an even stronger result, to wit, there is a $\Pi^0_1$-formula $\rho(x)$
that is extensional and $\Pi^0_1$-flexible.

We leave one important question open: what
happens when we weaken Extensionality to Consistent Extensionality, i.e., Extensionality for consistent extensions?
Consistent Extensionality is between full Extensionality and Conditional Extensionality.
\end{abstract}

\maketitle

{\Large
\textcolor{uuxred}{This preprint is superseded by preprint ArXiv:2506.13524, \emph{Extensional Independence}, by Taishi Kurahashi and Albert Visser}
}

\section{Introduction}
In his paper \cite{hamk:nonl22},  Joel Hamkins asks the following question (Question 28).\footnote{For notational consistency
with the rest of this notes we changed the variable-names in the formulation of the question.}
Is there a $\Pi^0_1$-formula $\rho(x)$ with the following properties?
\begin{description}
\item[Independence]  If ${\sf PA}+ \phi$ is consistent, then so are ${\sf PA}+ \phi+\rho(\gnum\phi)$ and\\
 ${\sf PA}+ \phi+\neg\, \rho(\gnum\phi)$.
\item[Extensionality] If ${\sf PA} \vdash \phi \iff \psi$, then  $ {\sf PA} \vdash\rho(\gnum \phi) \iff \rho(\gnum\psi)$.
\end{description}

\noindent
We call a formula that satisfies both Independence and Extensionality \emph{an extensional Rosser formula}.
We allow an extensional Rosser formula to be of any complexity.

We answer Hamkins' question negatively. There simply is no extensional Rosser formula, not just over {\sf PA},
but over a wide range of theories.

We will show that, if we weaken Extensionality to Conditional Extensionality, Hamkins' question gets a positive
answer. What is more, in this case, we can strengthen \emph{independence} to \emph{$\Pi^0_1$-flexibility}. Specifically, we prove
that there is a $\Pi^0_1$-formula $\rho(x)$ over {\sf PA} with the following properties.
\begin{description}
\item[$\Pi^0_1$-Flexibility] If ${\sf PA}+ \phi$ is consistent, then, for all $\Pi^0_1$-sentences $\pi$, the theory ${\sf PA}+ \phi+(\rho(\gnum\phi)\iff \pi)$ is
consistent.
\item[Conditional Extensionality] If ${\sf PA} \vdash \phi \iff \psi$, then $ {\sf PA}+ \phi \vdash\rho(\gnum \phi) \iff \rho(\gnum\psi)$.
\end{description}

\noindent
The proof of our result is an adaptation of the methods of \cite{viss:abso21}.

The notion of \emph{flexibility} was introduced by Saul Kripke in \cite{krip:flex62}. 
 We will call our formula $\rho(x)$ \emph{an extensional $\Pi^0_1$-reflexive formula}. 
 
 \begin{remark}{\footnotesize
 Kripke also studies flexibility not just for sentences but also
for formulas. This introduces an ambiguity. Our $\rho(x)$ is $\Pi^0_1$-flexible in
the sense that certain  sentential instances are $\Pi^0_1$-flexible sentences. Thus,
 \emph{parametrised sentence} would perhaps
be the better terminological choice for our notion over \emph{formula}. To avoid clumsy wording
we will persist in using \emph{flexible formula}. The reader will have to keep the context of the paper
in mind to get the right reading.}
\end{remark}
 
 Our result leaves the following question open.
 
 \begin{question}\label{vragendesmurf}
 Is there a Rosser formula that satisfies the following weakened form of Extensionality.
 
 \begin{description}
 \item[Consistent Extensionality]
 Suppose ${\sf PA}\nvdash \neg\, \phi$ and ${\sf PA} \vdash \phi\iff \psi$.
 Then, we have ${\sf PA}\vdash \rho(\gnum\phi) \iff \rho(\gnum\psi)$.
 \end{description}
 See also Remark~\ref{remarkablesmurf}.
 \end{question}
 
 \noindent We feel that Hamkins' problem is only partially solved as long as Question~\ref{vragendesmurf}
 remains open.
  
  \section{There are no Extensional Rosser Formulas}
  
  In this section, we prove that there is no formula that is both independent and
  extensional, thus answering
   Hamkins' original question in the negative.
  
  \begin{theorem}\label{grotesmurf}
  Let $U$ be any consistent theory that extends the Tarski-Mostowski-Robinson theory {\sf R}.
  Then, there are no extensional Rosser formulas over base theory $U$.
  \end{theorem}
  
  \noindent
  Note that we do not need any constraint on the complexity of the axiom set of $U$.
  
  \begin{proof}
  Suppose $\rho(x)$ is an extensional Rosser formula over $U$. We form the fixed points
  $\vdash \phi_0 \iff \rho(\gnum {\phi_0})$ and $\vdash \phi_1 \iff \neg\, \rho(\gnum {\phi_1})$.
  By Independence, we find (a) $\vdash \neg\,\phi_0$ and $\vdash \neg\,\phi_1$. 
  So,  $\vdash \phi_0 \iff \bot$ and $\vdash \phi_1 \iff \bot$. Thus, by Extensionality, (b)
  $\vdash \rho(\gnum{\phi_0}) \iff \rho(\gnum \bot)$ and  $\vdash \rho(\gnum{\phi_1}) \iff \rho(\gnum \bot)$.
  By the Fixed Point Equations in combination with (b), we find (c) $\vdash \phi_0 \iff \rho(\gnum \bot)$ and 
  $\vdash \phi_1 \iff \neg\, \rho(\gnum \bot)$. By (a) and (c), we find
  $\vdash \neg\,\rho(\gnum\bot)$ and $\vdash \rho(\gnum\bot)$, contradicting the consistency of $U$.
  \end{proof}
  
  \noindent By minor adaptations of the formulation and the proof, we find a similar result for theories that \emph{interpret}
  {\sf R}.

  \begin{remark}\label{remarkablesmurf}
  The above argument only uses the special case of Extensionality for inconsistent formulas.
  So, one may wonder what happens if we simply ban this case. 
  This is our Question~\ref{vragendesmurf}.
  
  We note that Consistent Extensionality implies Conditional Extensionality, for which we have our positive
  result. Thus, Consistent Extensionality is in the gap between our negative answer and our positive result.
  \end{remark}
  
  \section{G\"odel's Incompleteness Theorems and Flexibility}
  We briefly revisit G\"odel's Incompleteness Theorems and discuss their connection with $\Pi^0_1$-flexibility.

  \subsection{A Salient Consequence}\label{g-smurf}
  We formulate the for our paper  salient consequence of (the proofs of) G\"odel's Incompleteness Theorems.
  Let $U$ be a c.e. extension of {\sf EA}.\footnote{We can work with a much weaker theory then {\sf EA}. However, in the present paper,
  we do not strive for greatest generality.}
  We can find a $\Pi^0_1$-formula $\gamma$ such that:
  \begin{description}
  \item[Independence]
  Suppose $U+\phi$ has property $\mc X$. Then, both $U+\phi+\gamma(\gnum \phi)$ and $U+\phi+\neg\,\gamma(\gnum \phi)$
  are consistent.
  \item[Monotonicity]
  Suppose $U \vdash \phi \to \psi$. Then, $U \vdash \gamma(\gnum\phi) \to \gamma(\gnum \psi)$.
  \end{description}
  
  \noindent
  What can $\mc X$ be? It cannot be consistency by Theorem~\ref{grotesmurf}. Here are some possibilities.
  \begin{enumerate}[I.]
  \item
  $U+\phi$ is a true arithmetical theory.
  \item
  $U+\phi$ is $\omega$-consistent.
  \item
  $U+\phi$ is 1-consistent, in other words, $U+\phi$ satisfies $\Sigma^0_1$-reflection.
  \item
  $U+\phi \nvdash \neg\, \gamma(\gnum\phi)$.
  \end{enumerate}
  
  \noindent 
  We briefly discuss the ins and outs of these choices in Appendix~\ref{donaldsmurf}.
  
  Joel Hamkins in \cite{hamk:nonl22} shows that that Independence and Monotonicity cannot
  be combined under very general conditions. His result also covers a conditional version
  of Monotonicity. See also Appendix~\ref{gastsmurf}.

   \subsection{Flexibility}
 Saul Kripke in his paper \cite{krip:flex62} proves the existence of flexible predicates for a wide variety of
formal theories. Let $U$ be a consistent c.e. extension of {\sf EA}. We define:
\begin{itemize}
\item
$\chi$ is a $\Gamma$-flexible sentence (over $U$) iff,
for all $\Gamma$-sentences $\chi'$, we have $\chi \iff \chi'$ is consistent with $U$.
\end{itemize}

We see that independence is simply $\verz{\top,\bot}$-flexibility. So,
we can view $\Pi^0_1$-flexibility as a generalisation of the Rosser property.

Interestingly, in his paper, Kripke does not discuss the flexibility of the consistency statement.
The reason is undoubtedly that he was interested in a Kleene-style proof of a far more general theorem. 

 We work over Elementary Arithmetic {\sf EA} aka $\mathrm I\Delta_0+{\sf Exp}$.
  Let us first fix some concepts. Let ${\sf proof}(x,y)$ be some standard elementary arithmetisation of `$x$ is a proof from assumptions  of $y$'.
  Let the elementary formula ${\sf ass}(x,y)$ stand for `$y$ is an assumption that occurs in $y$'. As usual, we assume that ${\sf ass}(x,y)$ implies $y<x$ over
  {\sf EA}. We define:
  \begin{itemize}
  \item
   ${\sf proof}_{\tupel\alpha}(x,y) :\iff {\sf proof}(x,y) \wedge \forall z < x\, ({\sf ass}(x,z) \to \alpha(z))$.
   \item
    $\opr_{\tupel\alpha}\phi :\iff \exists x\, {\sf proof}_{\tupel\alpha}(x,\gnum \phi)$. 
    \item
    As usual we write $\oco$ for $\neg\opr\neg$.
  \item
   $\opr_{\tupel{\alpha}+\phi}\psi :\iff \opr_{\tupel{\alpha'}}\psi$, where $\alpha'(x) \iff (\alpha(x) \vee x= \gnum\phi)$.
   \end{itemize}
   
   \noindent
   We say that $\alpha$ \emph{semi-represents} an axiom set $\mc A$ of $U$ iff, for each $\phi \in \mc A$, we have $U \vdash \alpha(\gnum \phi)$.

We prove a simple result.

\begin{theorem} \label{smurfigesmurf}
Suppose  $U$ is a c.e. extension of {\sf EA}. Let $\alpha$ be an elementary formula that semi-represents an axiom set of $U$.
Let $\sigma$ be any $\Sigma^0_1$-sentence. Suppose $U \vdash \opr_{\tupel\alpha}\bot \iff \neg\, \sigma$.
Then, $U \vdash \opr_{\tupel\alpha} \bot$. Thus, if  $U \nvdash \opr_{\tupel\alpha} \bot$, then $\oco_\alpha\top$ is $\Pi^0_1$-flexible over $U$.
\end{theorem}

\begin{proof}
By verifiable $\Sigma^0_1$-completeness,
we find $U+ \sigma \vdash \oco_{\tupel\alpha+\sigma} \top$. Hence, by the Second Incompleteness Theorem,
$U \vdash \neg\, \sigma$ and, thus, $U \vdash \opr_{\tupel\alpha} \bot$.
\end{proof}

Theorem~\ref{smurfigesmurf} also follows from the well-known fact that the inconsistency statement is
$\Pi^0_1$-conservative.\footnote{Per Lindstr\"om, in \cite[p94]{lind:aspe03}, ascribes this to Georg Kreisel in \cite{kreis:weak62}.} in combination with the following simple insight. 

\begin{theorem}\label{saaiesmurf}
Suppose $\neg\,\phi$ is $\Gamma$-conservative over $U$ and
$U + \phi$ is consistent. Then, $\phi$ is $\Gamma$-flexible.
\end{theorem}

\begin{proof}
Suppose $\neg\,\phi$ is $\Gamma$-conservative over $U$ and $\chi \in \Gamma$ and $\phi\iff \chi$ is inconsistent with $U$.
Then, $U \vdash \neg\,\phi \iff \chi$. Hence, by $\Gamma$-conservativity, $U \vdash \chi$. So, $U \vdash \neg\, \phi$.
\end{proof}

\noindent Curiously, Theorem~\ref{saaiesmurf} does not need any constraint on $U$. Its axiom set can have any complexity.
It need not have numerals.

\begin{remark}
{\footnotesize We note that the proof of Lemma~\ref{smurfigesmurf} is not constructive. It is easily shown that this is essential.
The condition for $\Sigma^0_1$-flexibility is in the constructive case $U \nvdash \neg\neg\,\opr_{\tupel\alpha}\bot$.
We note that the negation of this last condition constitutes a counterexample since  $\neg\neg\,\opr_{\tupel\alpha}\bot$
is equivalent to $\neg\, (\neg\,\opr_{\tupel\alpha}\bot \iff \top)$.}
\end{remark}

%


\begin{remark}{\footnotesize
Under the circumstances of Theorem~\ref{smurfigesmurf}, let $\nu$ be a $\Pi^0_1$-Rosser sentence of $U$ for $\opr_{\tupel{\alpha}}$.
Then, we have: if $U \vdash \nu \iff \sigma$, then $U \vdash \opr_{\tupel{\alpha}}\bot$. 

Conversely, if $U \vdash \opr_{\tupel{\alpha}}\bot$, then
$U \vdash \nu \iff \neg\, \nu^\bot$, where $\nu^\bot$ is the opposite of $\nu$. Thus, the Rosser construction does deliver an independent sentence
for all consistent extensions of {\sf EA}, but not a $\Pi^0_1$-flexible
sentence for some consistent extensions of {\sf EA}.}
\end{remark}

\section{There is a Conditionally Extensional  $\Pi^0_1$-Flexible Formula}
We develop our positive result that there is a conditionally extensional $\Pi^0_1$-flexible formula for
  c.e. base theories extending {\sf PA} and, more generally, for essentially reflexive, sequential c.e. theories.
  We take a consistent arithmetical extension $U$ of {\sf PA} as a base. Let $\alpha$ be an elementary formula that
  gives an axiom set $\mc A$ of $U$. Specifically, we demand that if $\psi\in \mc A$, then $U \vdash \alpha(\gnum\psi)$ and
   if $\psi\not\in \mc A$, then $U \vdash\neg\, \alpha(\gnum\psi)$.
 It easy to see how the results extend to essentially reflexive sequential  c.e. theories.
 This only requires slight adaptations of the formulations.
 
  The sentences we provide will be consistency statements of theories of slow provability.
  
  \subsection{History}
  Hamkin's question is inspired by the following result of Volodya Shavrukov and Albert Visser in \cite{shav:unif14}.
\begin{theorem}[Shavrukov, Visser]
There is a $\Delta^0_2$-formula $\rho(x)$ over {\sf PA} with the following properties.
\begin{description}
\item[Independence] If ${\sf PA}+ \phi$ is consistent, then so are ${\sf PA}+ \phi+\rho(\gnum\phi)$ and\\
 ${\sf PA}+ \phi+\neg\, \rho(\gnum\phi)$.
\item[Conditional Extensionality] If ${\sf PA} \vdash \phi \iff \psi$, then $ {\sf PA}+ \phi \vdash\rho(\gnum \phi) \iff \rho(\gnum\psi)$.
\end{description}
\end{theorem}

\noindent
We call a formula that satisfies both Independence and Conditional Extensionality \emph{a
conditionally extensional Rosser formula.} We note that Shavrukov {\&} Visser's result can be rephrased as follows.
For any formula $\theta(x)$, let $\theta^+_\phi$ be $\phi \wedge \theta(\gnum\phi)$.
There is a $\Delta^0_2$-formula $\rho(x)$ over {\sf PA} such that:
\begin{description}
\item[Independence] If ${\sf PA}+ \phi$ is consistent, then so are ${\sf PA}+ \phi + \rho^+_\phi$ and\\
 ${\sf PA}+ \phi+\neg\, \rho^+_\phi$.
\item[Extensionality] If ${\sf PA} \vdash \phi \iff \psi$, then $ {\sf PA} \vdash\rho^+_\phi \iff \rho^+_\psi$.
\end{description}

\noindent In other words, Conditional Extensionality for $\rho$ entails Extensionality for $\rho^+$.

Relative to Shavrukov {\&} Visser's result, we see that Hamkins is asking for two improvements:
(i) bring down $\Delta^0_2$  to $\Pi^0_1$ and (ii) improve Conditional Extensionality to
Extensionality. We have seen that (ii) cannot be fulfilled.  We will show that (i) can be satisfied, when we keep
Conditional Extensionality.
 
 The formula employed by Shavrukov {\&} Visser is a meaningful formula. It generates
 Smory\'nski's  Rosser sentences from \cite{smor:self89}.
 Moreover, it generates  fixed-point-free sentences that satisfy (uniquely modulo provable equivalence) the 
 G\"odel equation for a certain Feferman provability predicate.
 See  \cite{smor:self89} and \cite{shav:smar94}. 
 Thus, the generated sentences, while not being consistency statements, do share two important properties with consistency
 statements: being explicit, i.e., fixed-point-free, and being unique solutions of a G\"odel equation.
 
  It is easy to see that Fefermanian G\"odel sentences are both $\Pi^0_1$- and $\Sigma^0_1$-flexible.
 
 \medskip
 The  conditionally extensional $\Pi^0_1$-flexible $\Pi^0_1$-formula offered in this paper is an adaptation of the construction of 
  \cite{viss:abso21}. It  delivers slow consistency statements in the style of
 \cite{viss:abso21}. See also \cite{frie:slow13} for a more proof-theoretic take on 
 slow provability. The construction of the formulas is fixed-point-free.

\subsection{Uniform Provability}\label{prinsbsmurf}
We develop the basics of uniform provability  in this subsection.
We fix our base theory $U$ that {\sf EA}-verifiably extends {\sf PA}. The axiomatisation of $U$ will be given by
the elementary formula $\alpha$. We will suppress the subscript $\tupel\alpha$, thus writing e.g.
$\opr\psi$ for $\opr_{\tupel{\alpha}}\psi$ and $\opr_\phi\psi$ for $\opr_{\tupel{\alpha}+\phi}\psi$, etcetera.

 Even if
our main argument is about $U$,
a lot of the present development is available in Elementary Arithmetic, {\sf EA},
 aka $\mathrm I\Delta_0+{\sf Exp}$. 
 
 Uniform proofs are a minor modification of proofs. We will use uniform provability to define
 our notion of smallness and we will use smallness to define slow provability as provability from
 small axioms.
 
We define the relation $\preceq_n$ on sentences $\leq n$ as follows. We have:
$\phi \preceq_n \psi$ iff $\psi$ follows from $\phi$ plus the sentences 
provable in $U$ by proofs $\leq n$ by propositional logic. 
We use $\preceq_x$ for the arithmetization of $\preceq_n$.
We note that $\preceq_x$ will be {\sf EA}-provably reflexive and transitive.

 We define:
 \begin{itemize}
  \item
   $\opr_{\phi,(x)}\psi :\iff \exists p\bleq x\, {\sf proof}_\phi(p,\gnum\psi)$,
\item
$\gopr_{\phi,(x)}\psi :\iff \exists \chi \bleq x \, (\phi \preceq_x \chi \wedge  \opr_{\chi,(x)} \psi)$,
\item
$\gopr_\phi\psi :\iff \exists x\, \gopr_{\phi,(x)} \psi$.
  \end{itemize}
  
  The grey box stands for \emph{uniform provability}. We note that the uniformity is w.r.t the formula in the subscript.
  The next theorem explicates the relationship between uniform provability and provability.
  
  \begin{theorem}\label{mageresmurf}
  \begin{enumerate}[1.]
  \item
  ${\sf EA} \vdash \opr_{\phi,(x)} \psi \to \gopr_{\phi,(x)}\psi$, 
  \item
  There is an elementary function $F$ such that
  ${\sf EA} \vdash \gopr_{\phi,(x)} \psi \to \opr_{\phi,(Fx)}\psi$, 
  \item
  ${\sf EA} \vdash \opr_\phi\psi \iff \gopr_\phi\psi$.
  \end{enumerate}
  \end{theorem}
  
  \noindent We leave the simple proof to the reader.
  The next theorem gives us the desired uniformity property.
  
  \begin{theorem}\label{bravesmurf}
  \begin{enumerate}[i.]
  \item
  Suppose $p$ is the code of a $U$-proof of $\phi \to \psi$.
  Then, \[{\sf EA}\vdash \forall x\bgeq \num p\,\forall \chi \bleq x\,(\gopr_{\psi,(x)}\chi \to \gopr_{\phi,(x)}\chi).\]
  \item
    Suppose $p$ is the code of a $U$-proof of $\phi \iff \psi$.
  Then, \[{\sf EA}\vdash \forall x\bgeq \num p\,\forall \chi \bleq x\,(\gopr_{\psi,(x)}\chi \iff \gopr_{\phi,(x)}\chi).\]
  \end{enumerate}
  \end{theorem}
 
\begin{proof}
We prove (i). Case (ii) is similar. Suppose $p$ is the code of a $U$-proof of $\phi \to \psi$.
It follows that $\phi \preceq_p \psi$ and, hence, that ${\sf EA} \vdash \forall x\bgeq \num p\; \phi \preceq_x \psi$.

We reason in {\sf EA}. Suppose $x\geq \num p$ and $\gopr_{\psi,(x)}\chi$.
So, for some $\theta$, we have $ \theta \leq x$ and  $\psi \preceq_x \theta$ and $\opr_{\theta,(x)} \chi$.
We also have $\phi \preceq_x \psi$ and, hence, $\phi \preceq_x\theta$. It follows that
 $\gopr_{\phi,(x)}\chi$
\end{proof}
 
 \subsection{Uniform Smallness}
We use $\sigma$, $\sigma'$, \dots\ to range over (G\"odel numbers of) $\Sigma^0_1$-sentences. 
Let $\phi$ be any arithmetical sentence. 
We define \emph{uniform $\phi$-smallness} as follows:
\begin{itemize}
\item
$ \mf S_\phi(x) : \iff \forall \sigma\bleq x\, (\gopr_{\phi,(x)} \sigma \to {\sf True}_{\Sigma^0_1}(\sigma))$.
\end{itemize}

It is easily seen that, modulo $U$-provable equivalence, uniform $\phi$-smallness is $\Sigma^0_1$.
We already have this result over ${\sf EA}+\mathrm B\Sigma_1$.\footnote{We can tinker with the
definition of smallness to avoid the use of collection. See \cite{viss:abso21}.}
We will suppress the `uniform' in the rest of this paper, since we only consider the uniform case. 
  
Trivially, $\phi$-smallness is downwards closed.

If $\phi$ is consistent with $U$, then it is consistent with $U+\phi$ that not all numbers 
are small, since $U+\phi$ does not prove $\Sigma^0_1$-reflection for
$(U+\phi)$-provability. On the other hand, we have:

\begin{theorem}\label{kleinesmurf}
For every $n$, the theory $U+\phi$ proves that $n$ is $\phi$-small, i.e., 
$U+\phi\vdash \mf S_\phi(\underline n)$. Moreover, {\sf EA} verifies this insight, i.e.,
 ${\sf EA} \vdash \forall x\, \opr_\phi \mf S_\phi(x)$.
\end{theorem}

The principle articulated in the formalised part of the theorem is a typical example of an \emph{outside-big-inside-small principle}. 
Objects that may be  very big in the outer world are  small in the inner world.

\begin{proof}
Consider any number $n$. We work in {\sf PA}. Suppose $\sigma \leq \num n$ and $\gopr_{\phi, (\num n)}\sigma$.
Then $\sigma$ is standard (i.o.w., we can replace the bounded existential quantifier by a big disjunction),  and
 $\opr_{\phi, (F(\num n))}\sigma$. So, by small reflection, we have $\sigma$, and, hence, ${\sf True}_{\Sigma^0_1}(\sigma)$.

This simple argument can clearly be verified in {\sf EA}.
\end{proof}

\begin{theorem}\label{oppervlakkigesmurf}
\begin{enumerate}[1.]
\item
Suppose $U\vdash \phi \to \psi$. Then, $U +\phi \vdash \forall x\, ({\mf S}_\phi(x) \to {\mf S}_\psi(x))$. 
\item
Suppose $U\vdash \phi \iff \psi$. Then, ${\sf PA} +\phi \vdash \forall x\, ({\mf S}_\phi(x) \iff {\mf S}_\psi(x))$. 
\end{enumerate}

\smallskip\noindent
These results can be verified in {\sf EA}.
\end{theorem}

\begin{proof}
Ad (1): Suppose $U\vdash \phi \to \psi$. Let $p$ a code of a
  $U$-proof of $( \phi \to \psi)$. 
  We reason in $U+\phi$. 
  
  Suppose $x < \num p$.
In that case, we have both  ${\mf S}_\phi(x)$ and ${\mf S}_\psi(x)$, since $x$ is standard and we have both $\phi$ and $\psi$ and, thus,
we may apply Theorem~\ref{kleinesmurf}.

Suppose $x \geq \num p$. In this case, we are immediately done  by Theorem~\ref{bravesmurf}.

Case (2) is similar. 
\end{proof}

\subsection{Uniform Slow Provability}
This section gives our main argument. The argument is an adaptation of the argument in
\cite{viss:abso21}.

We define the \emph{uniform slow $(U+\phi)$-provability} of $\psi$ or $\apr_\phi \psi$ as: $\psi$ is provable from $\phi$-small $(U+\phi)$-axioms. 
We give the formal definition.
\begin{itemize}
\item
$\opr_{\phi,x}\psi$ iff there is a proof of $\psi$ from $(U+\phi)$-axioms $\leq x$.
\item
$ \apr_\phi \psi :\iff \exists x\, (\opr_{\phi,x}\psi \wedge {\mf S}_\phi(x))$.
\end{itemize}

We note that, in case $\phi$ is consistent with $U$ the internally defined set of small axioms numerates the $(U+\phi)$-axioms over $U+\phi$.
Let \[\beta(x) := ((\alpha(x)\vee x=\gnum\phi) \wedge \mf S_\phi(x)).\] Then, $\apr_\phi \chi$ is $\opr_{\tupel\beta}\chi$. 
Thus, $\apr_\phi$ is a Fefermanian predicate based on an enumeration of the axioms of
$U+\phi$. 

We see that the  axiom set is $\Sigma^0_1$, so, modulo $U$-provable equivalence or even modulo $({\sf EA}+\mathrm B\Sigma_1)$-provable equivalence, 
$\apr_\phi$ is $\Sigma_1^0$. 
Here   ${\sf EA}+\mathrm B\Sigma_1$ is Elementary Arithmetic plus $\Sigma^0_1$-collection.

\begin{remark}
\emph{Caveat emptor:}
 $\apr_\phi\psi$ is \emph{not} the same as $\apr_\top(\phi \to \psi)$.
 See Appendix~\ref{gastsmurf}.
 \end{remark}

\begin{theorem}\label{loebsmurf}
$\apr_\phi$ satisfies the L\"ob Conditions over $U+\phi$.
This result can be verified in {\sf EA}.
\end{theorem}

\begin{proof}
All standard axioms of $U+\phi$ are $(U+\phi)$-provably small.
It follows that we have L\"ob's Rule, aka Necessitation.

We note that, by Necessitation, we have {\sf EA} inside $\apr_\phi$ according to 
$U+\phi$.

It is immediate that $U+\phi \vdash (\apr_\phi\psi \wedge \apr_\phi (\psi\to\chi)) \to \apr_\phi\chi$.

Clearly, $U+\phi$ proves that $\apr_\phi\psi$ is equivalent to a $\Sigma^0_1$-sentence, say $\sigma$.
Moreover, by necessitation, we have $U+\phi \vdash \apr_\phi(\apr_\phi\psi \iff \sigma)$. 
It follows that:\qedright
\begin{eqnarray*}
U+ \phi \vdash \apr_\phi\psi & \to & \sigma\\
& \to & \apr_\phi \sigma \\
& \to & \apr_\phi \apr_\phi \psi
\end{eqnarray*}
 \end{proof}
 
 \begin{remark}\label{hacksmurf}
Suppose we arrange that every proof is larger than an induction axiom that implies ${\sf EA}+\mathrm B\Sigma_1$.
Then, the axiom-set for $\apr_\phi$ will ${\sf EA}+\mathrm B\Sigma_1$-verifiably contain ${\sf EA}+\mathrm B\Sigma_1$.

So, we will have L\"ob's Rule, aka Necessitation,  over ${\sf EA}+\mathrm B\Sigma_1$. Note that we do not need to have L\"ob's Rule
over $U$ itself.

Also, the axiom set will be $\Sigma^0_1$ modulo provable equivalence in the $\apr_\phi$-theory. This will give us full L\"ob's logic
over ${\sf EA}+\mathrm B\Sigma_1$.
\end{remark}

\begin{theorem}\label{extensionalitysmurf}
Suppose $U \vdash \phi \iff \psi$. Then, $U+\phi \vdash \forall \chi\, (\apr_{\phi}\chi \iff \apr_\psi \chi)$.
This result can be verified in {\sf EA}.
\end{theorem}

We will assume that codes of formulas occurring in a proof are smaller than the code of that proof.

\begin{proof}
Suppose $p$ is a code of a $U$-proof of  $ \phi \iff \psi$. 
We reason in $U+\phi$. We also will have $\psi$.
Consider any $\chi$ and suppose $q$ witnesses that $\apr_\phi\chi$. Thus, $q$ is  a proof of $\chi$ from $\phi$ plus $\phi$-small $U$-axioms. 

 We note that $\psi$ itself is a $\psi$-small axiom, since it is standard.
Moreover, all the $U$-axioms used in $\num p$ are standard, since $\num p$ is.
So $\apr_\psi\psi$ and $\apr_\psi (\phi\iff \psi)$.  Hence,  $\apr_\psi\phi$. Say this is witnessed by $r$.
Thus, we can transform $q$ to a witness $q'$ of $\apr_\psi\chi$, by replacing $\phi$ as an axiom by the
proof $r$. We note that $\phi$-small $U$-axioms used in $q$ are also $\psi$-small  by Theorem~\ref{oppervlakkigesmurf}.
 It follows that $\apr_\psi\chi$.

The other direction is similar.
\end{proof}

\begin{theorem}[Emission]\label{emission}
${\sf EA}\vdash \opr_\phi \psi \to \opr_\phi \apr_\phi \psi$.  
\end{theorem}

\begin{proof}
We reason in {\sf EA}. Suppose $\opr_\phi\psi$. 
Then, clearly, for some $x$, we have $\opr_{\phi,x} \psi$.
Hence, $\opr_\phi \opr_{\phi,x}\psi$. Also, Theorem~\ref{kleinesmurf} gives us $\opr_\phi\mf S_\phi(x)$.
So, $\opr_\phi (\opr_{\phi,x}\psi \wedge \mf S_\phi(x))$ and, thus, $\opr_\phi\apr_\phi \psi$.
\end{proof}

 \begin{theorem}[Absorption]\label{absorption}
${\sf EA}  \vdash \opr_\phi\apr_\phi \psi \to \opr_\phi \psi$.  
 \end{theorem}
 
 The proof turns out to be remarkably simple. 
 
 \begin{proof}
We find $\nu$ such that ${\sf EA} \vdash \nu \iff  (\exists x \, \opr_{\phi,x} \psi ) < \opr_\phi \nu$.
 We note that  $\nu$ is $\Sigma^0_1$.

We reason in {\sf EA}. 
Suppose $\opr_\phi\apr_\phi \psi$. We prove $\opr_\phi \psi$.

We reason inside $\opr_\phi$. We have  $\apr_\phi \psi$.
 So, for some $x$, (i) $\opr_{\phi,x} \psi$ and (ii) $x$ is $\phi$-small.
In case not $\opr_{\phi,(x)} \nu$, by (i) and the fixed point equation, we find $\nu$. 
If we do have $\opr_{\phi,(x)} \nu$, we find $\nu$ by (ii)
and the fact that $\phi$-small proofs are $\Sigma^0_1$-reflecting
(by Theorem~\ref{mageresmurf}(1)).
 We leave the $\opr_\phi$-environment.
We have shown $\opr_\phi \nu$. 

It follows, (a) that for some $p$, we have $\opr_\phi \opr_{\phi,(p)} \nu$ and, by the fixed point equation 
for $\nu$, (b) $\opr_\phi ((\exists x\, \opr_{\phi,x} \psi) < \opr_\phi \nu)$. 
Combining (a) and (b), we find $\opr_\phi\opr_{\phi,p} \psi$, and, thus,
since $U+\phi$ is, {\sf EA}-verifiably, essentially reflexive, we obtain  $\opr_{\phi} \psi$, as desired. 
\end{proof}

\noindent
We provide some notes on the proof of Theorem~\ref{absorption} in Appendix~\ref{emisabso}.

\begin{theorem}\label{maintheorem}
We have:
\begin{description}
\item[$\Pi^0_1$-Flexibility] If $U+ \phi$ is consistent, then, for all $\Pi^0_1$-sentences $\pi$, we have $U+ \phi+ (\aco_\phi \top \iff \pi)$ 
is consistent.
\item[Conditional Extensionality] If ${\sf PA} \vdash \phi\iff \psi$, then  $ {\sf PA}+ \phi \vdash \aco_\phi\top \iff \aco_\psi\top$.
\end{description}

\smallskip\noindent
These results can be verified in {\sf EA}.
\end{theorem}

\begin{proof}
Suppose $U+ \phi$ is consistent. Then, 
$U+\phi \nvdash \apr_\phi\bot$, by the meta-version of Theorem~\ref{absorption}.
We  apply Theorem~\ref{smurfigesmurf} to obtain $\Pi^0_1$-Flexibility.

Conditional Extensionality follows by Theorem~\ref{extensionalitysmurf}.
\end{proof}


\appendix

\section{Sufficient Conditions for G\"odel's Incompleteness Theorems}\label{donaldsmurf}
We discuss the conditions given in Subsection~\ref{g-smurf}, to wit:
  \begin{enumerate}[I.]
  \item
  $U+\phi$ is a true arithmetical theory.
  \item
  $U+\phi$ is $\omega$-consistent.
  \item
  $U+\phi$ is 1-consistent, in other words, $U+\phi$ satisfies $\Sigma^0_1$-reflection.
  \item
  $U+\phi \nvdash \neg\, \gamma(\gnum\phi)$.
  \end{enumerate}

   As is well known, G\"odel avoided Condition I, since, around the time of his discovery, the notion of
  truth was considered somewhat suspect. He choose $\omega$-consistency. Somehow, this condition had a long life and occurs
  even in recent textbooks. It seems that 1-consistency is a better condition because it allows greater generality.
  Since $\Sigma^0_1$-truth is arithmetically definable it can hardly be suspect. 
  
  But why not go on and
  opt for Condition IV? That gives us a necessary and sufficient condition and, thus, is the only truly
  motivated condition. There are two objections. Firstly, it is trifling, stipulating half of what we wanted to
  prove. Secondly, we need coding and the like to define $\gamma$ at all and so the condition is not sufficiently natural.
  In spite of these objections, I am a fan of Condition IV. It may seem trifling, but it is less so when we
  consider more general statements like the $\Pi^0_1$-flexibility of the G\"odel sentence and of the consistency statement.
  Moreover, it is well-known that we can, modulo {\sf EA}-provable equivalence, eliminate the arbitrary choices
  in the definition of the consistency statement. See e.g.\ \cite{viss:seco11}. Finally, when one considers the
  Incompleteness Theorems for slow provability, Conditions I--III  play no role at all.
  The unprovability of inconsistency follows from the choice of the representation of the axiom set.

\section{Failure of Monotonicity Revisited}\label{gastsmurf}

We revisit Hamkins' argument that independence and monotonicity cannot be combined.
We work over any consistent theory $U$. The axiom set of the theory may have any complexity.
We need not ask anything of the theory, not even that it contains numerals.
Suppose we have a mapping $\phi \mapsto \rho_\phi$ from $U$-sentences to $U$-sentences. There are no assumptions on the
complexity of the mapping or on the complexity of the values and the like. Suppose we have \emph{Independence:} if $U+\phi$ is consistent, then
$\nvdash \phi \to \rho_\phi$ and $\nvdash \phi \to \neg\, \rho_\phi$ and that we have
\emph{Monotonicity:} if $\vdash \phi \to \psi$, then $\vdash \rho_\phi \to \rho_\psi$.

We have $\vdash \neg\, \rho_\top \to \top$. Hence, $\vdash \rho_{\neg \,\rho_\top} \to \rho_\top$.
Ergo,  $\vdash  \neg\, \rho_\top \to \neg\,  \rho_{\neg \,\rho_\top} $. But, then it would follow by Independence that 
$\vdash \rho_\top$. Again by Independence, this is impossible.

Now let $U$ be a base theory as in Section~\ref{prinsbsmurf}.
When we take  $\rho_\phi := \aco_\phi \top$, since we have Independence, we find:
\[\nvdash  \neg\, \aco_\top\top \to \neg\,  \aco_{\neg \,\aco_\top\top}\top.\]
Rewriting this, we obtain: $\nvdash   \apr_\top\bot \to   \apr_{\apr_\top\bot}\bot$.
On the other hand, we do have: $\vdash   \apr_\top\bot \iff   \apr_\top \neg\, \apr_\top\bot$.
So, $\apr_\top \bot \nvdash   \apr_\top \neg\, \apr_\top\bot \to   \apr_{\apr_\top\bot}\bot$. So, we have a counterexample
 to $\phi \wedge \psi\vdash \apr_\phi (\psi \to \chi) \to \apr_{\phi\wedge \psi} \chi$, with $\phi := \top$, $\psi := \apr_\top\bot$ and
 $\chi := \bot$.
 
 In contrast, we have:
 
\begin{theorem}
$U+ (\phi\wedge\psi)\vdash \apr_{\phi\wedge \psi} \chi \to \apr_\phi (\psi \to \chi)$.
\end{theorem}

\begin{proof}
We reason in $U$ plus $\phi\wedge \psi$. Suppose $p$ is a proof of $\chi$ from
$\phi \wedge \psi$ and $(\phi\wedge\psi)$-small axioms. Then,
we can find a proof $q$ of $(\psi\to \chi)$ from $\phi$ plus $(\phi\wedge \psi)$-small axioms.
By Theorem~\ref{oppervlakkigesmurf}, we find that  $(\phi\wedge \psi)$-small axioms are also
$\phi$-small. So, $q$ witnesses $\apr_\phi (\psi \to \chi)$
\end{proof}

\section{Arrow Notation}
It is attractive to represent $\apr_\phi\psi$ as a sort of implication $\phi \tto \psi$.
This is not entirely comfortable since we do not have
the transitivity of implication. Anyway, to see how the alternative notation looks,
we provide the principles we derived before and some new ones ---without any
claim of completeness. 

We will assume that we always have the conjunction of a finite axiomatisation of 
${\sf EA}+\mathrm B\Sigma_1$ as an axiom, whether we have $\phi$ or not, as in 
Remark~\ref{hacksmurf}, so that we have $\Sigma^0_1$-completeness unconditionally
and also that our provability predicate is equivalent to a $\Sigma^0_1$-predicate.

\begin{itemize}
\item
If $\phi \vdash \psi$, then $\phi \vdash \phi \tto \psi$.
\item
$\vdash ((\phi \tto \psi) \wedge (\phi \tto (\psi \to \chi))) \to (\phi \tto \chi)$.
\item
$\phi \tto \psi \vdash \opr(\phi \tto \psi)$.
\item
$\opr \psi \vdash \phi \tto \opr\psi$.
\item
$(\psi \tto \chi) \vdash \phi \tto (\psi \tto \chi)$.
\item
$\ \opr(\phi \to \psi) \vdash \opr(\phi \to (\phi\tto \psi))$.
\item
$ \opr(\phi \to (\phi \tto \psi)) \vdash \opr(\phi \to \psi)$.
\item
$(\phi\wedge \psi),\, ((\phi\wedge \psi) \tto \chi) \vdash \phi \tto (\psi \to \chi)$.
\item
$\phi \tto ((\phi \tto \psi) \to \psi) \vdash  \phi \tto \psi$.
\end{itemize}

\section{On the Proof of Absorption}\label{emisabso}
The fixed point $\nu$ with ${\sf EA} \vdash \nu \iff  (\exists x \, \opr_{\phi,x} \psi ) < \opr_\phi \nu$
 is an instance of the fixed point used in the FGH Theorem, so called after Harvey Friedman, Warren Goldfarb and
 Leo Harrington who each, independently, discovered the argument associated with the fixed point.
 However, in fact, John Shepherdson discovered the argument first. See \cite{shep:rep61}.
 
 The second half of the proof of Theorem~\ref{absorption} is simply a proof of a specific version of the
 FGH theorem. We cannot quite follow the usual argument, since the opposite of $\nu$, to wit
 $  \opr_\phi \nu \leq (\exists x \, \opr_{\phi,x} \psi)$ is not $\Sigma^0_1$. 

The first half of the proof of Theorem~\ref{absorption} brings us from $\opr_\phi \apr_\phi \psi$ to $\opr_\phi \nu$.
The crucial point in the paper is the step where $\phi$-smallness is used is in moving from $\opr_{\phi,(x)}\nu$ to $\nu$.

What about using the fixed point $\apr^\ast_\phi \psi$ with
\[{\sf EA}\vdash \apr^\ast_\phi \psi \iff (\exists x \, \opr_{\phi,x} \psi ) < \opr_\phi\apr^\ast_\phi \bot \;?\]
That would deliver a Fefermanian provability predicate with absorption for the case that $\psi := \bot$.
We can manipulate this by modifying the definition using $\gopr\,$-trickery in order to get closer to extensionality. However, I do not
 see how to get full extensionality.

\end{document}